\documentclass[a4,11pt]{article}
\usepackage{graphicx}
\usepackage{amssymb}
\usepackage{url}
\usepackage{latexsym}
\usepackage{amssymb}
\usepackage{here}
\usepackage{wrapfig}
\usepackage{fancybox}
\usepackage{ascmac}
\usepackage{amsmath}
\usepackage{multicol}
\usepackage{mathtools}
\usepackage{amsthm}
\usepackage{color}
\usepackage{sidenotes}
\usepackage{cite}

\newcommand{\mc}{\mathcal}
\newcommand{\mb}{\mathbb}
\newcommand{\invnormzero}{\| L^{-1} \|}

\newcommand{\invnormXV}{\| L^{-1} \|_{{\mc L}(X, V)}}

\newcommand{\Honezero}{H^1_0(\Omega)}
\newcommand{\Hone}{H^1(\Omega)}

\newcommand{\Ltwo}{L^{2}(\Omega)}

\newcommand{\sqA}{A^{\frac{1}{2}}}
\newcommand{\sqAinv}{A^{-\frac{1}{2}}}
\newcommand{\ra}{\rightarrow}
\newcommand{\f}{\frac}

\newcommand{\minner}[2]{(#1,#2)_{M_{\sigma}}}
\newcommand{\mnorm}[1]{\| #1 \|_{M_\sigma}}

\newtheorem{Thm}{Theorem}[section]

\newtheorem{lem}[Thm]{Lemma}

\newtheorem{rem}[Thm]{Remark}

\def\vector#1{\mbox{\boldmath $#1$}}

\allowdisplaybreaks[4]

\begin{document}
\begin{center}
\noindent{\bf\Large Inverse norm estimation of perturbed Laplace operators and corresponding eigenvalue problems
}\vspace{10pt}\\
 {\normalsize Kouta Sekine$^{1,*}$, Kazuaki Tanaka$^{2}$, Shin'ichi Oishi$^{3}$}\vspace{5pt}\\
 {\it\normalsize $^{1}$Faculty of Information Networking for Innovation and Design, Toyo University, 1-7-11 Akabanedai, kita-ku, Tokyo 115-0053, Japan.\\
 $^{2}$Institute for Mathematical Science, Waseda University, 3-4-1, Okubo Shinjyuku-ku, Tokyo 169-8555, Japan\\
 $^{3}$Department of Applied Mathematics, Faculty of Science and Engineering, Waseda University, 3-4-1 Okubo, Tokyo 169-8555, Japan}
\end{center}
{\bf Abstract}. 
			In numerical existence proofs for solutions of the semi-linear elliptic system,
			evaluating the norm of the inverse of a perturbed Laplace operator plays an important role.
			We reveal an eigenvalue problem to design a method for verifying the invertibility of the operator and evaluating the norm of its inverse based on Liu's method and the Temple-Lehman-Goerisch method.
			We apply the inverse-norm's estimation to the Dirichlet boundary value problem of the Lotka-Volterra system with diffusion terms and confirm the efficacy of our method.
\renewcommand{\thefootnote}{\fnsymbol{footnote}}
\footnote[0]{{\it E-mail address:} $^{*}$\texttt{ k.sekine@computation.jp}\\[-3pt]}
\renewcommand\thefootnote{*\arabic{footnote}}
\section{Introduction}
	Numerical existence proofs for solutions of differential equations have been developed in recent decades and applied to various problems for which purely analytical methods have failed (see, for example, \cite{nakao2019numerical} and the references therein).
	Our objective is the boundary value problem of the system of elliptic differential equations	
	\begin{align}
	\label{main_system}
	\left\{
	\begin{array}{ll}
		-\Delta u_1 = f_1( u_1, u_2, \cdots, u_{m} ) & {\rm in~~} \Omega,\\
		-\Delta u_2 = f_2( u_1, u_2, \cdots, u_{m} ) & {\rm in~~} \Omega,\\
		~~~ \cdots & \\
		-\Delta u_{m} = f_m( u_1, u_2, \cdots, u_{m} ) & {\rm in~~} \Omega,\\
		u_1 = u_2 = \cdots = u_{m} =0 & {\rm on~~} \partial\Omega,
	\end{array}
	\right.
	\end{align}
	where $\Omega \subset {\mb R}^n~(n = 1, 2, 3)$ is a bounded domain  with a Lipschitz boundary,
	and $ f_i :  (H^1_0(\Omega))^m \rightarrow L^2(\Omega)$ ($i=1,2,\cdots,m$) are given operators.
	Here, for $ \Hone $ denoting the first order real $ L^2 $-Sobolev space,
	we define
	$ \Honezero := \{ u \in H^1(\Omega) : u = 0~\mbox{on} ~ \partial \Omega\} $ with inner product $(\cdot, \cdot)_{H^1_0}:=(\nabla \cdot, \nabla \cdot)_{L^2}$.
	For two Banach spaces $X$ and $Y$, ${\mc B}(X, Y)$ stands for the space of bounded linear operators from $X$ to $Y$ defined over the whole of $X$ with the norm $\| T \|_{{\mc B}(X, Y)} := \sup\{ \| T u \|_{Y} / \| u \|_{X} : {u \in X \setminus \{0\}}\}  $.
	We simply write ${\mc B}(X)={\mc B}(X, X)$.
	The expression ``$T:D\subset X \ra Y$'' is equivalent to ``an operator $T$ from $X$ to $Y$ is defined on $D \subset X$''.
	The domain of $T$ is denoted by $\mc{D}(T)$.
	When $\mc{D}(T)=X$, we simply write $T: X \ra Y$.

	To write \eqref{main_system} as an operator equation, we prepare required notation.
	Let $X := (L^2(\Omega))^m$ and $V := (\Honezero)^m$ be the Hilbert spaces with the inner products $(\cdot, \cdot)_{X} := \sum_{i=1}^m (\cdot, \cdot)_{L^2}$ and $(\cdot, \cdot)_{V} := \sum_{i=1}^m (\cdot, \cdot)_{H^1_0}$, respectively.
	We denote the topological dual of $V$ by $V^*$ with the duality product $\langle \cdot, \cdot \rangle$ for the pair $\{ V,  V^*\}$.
	We define ${\mathcal A}:  V \rightarrow V^*$ by 
	$\langle {\mathcal A} u, v \rangle := (u, v)_{V}$ for all $u,v \in V$
	and
	the densely-defined closed operator $A: {\mc D}(A) \subset X \rightarrow X$ by
	$A u = {\mathcal A} u,~u \in {\mc D}(A) := \{ u \in H^1_0(\Omega) : {\mathcal A} u \in X \}$.
	The operator $A$ is a positive self-adjoint operator with the compact inverse $A^{-1}: X \ra X$.
	Then, \eqref{main_system} is characterized as the operator equation
	\begin{align}
		\label{eq:operatoreq}
		A u = f(u),
	\end{align}
	where $f(u) = (f_1(u), \cdots, f_m(u))^T,~u \in V$ assumed to be Fr\'echet differentiable.
	In numerical existence proofs for \eqref{eq:operatoreq}, the following linear operators are considered:
	\begin{align*}
		{\mathcal L} = {\mathcal A} - Q &:~~ V \ra V^*,\\
		L = A - Q &:~~ {\mc D}({A}) \subset X \ra X,
	\end{align*}
	where $Q \in {\mc B}(V, X)$.
	Choosing $Q$ as $f'[\hat{u}]: V \rightarrow X $, the Fr\'echet derivative of $f$ at an approximate solution $\hat{u} \in V$ of \eqref{eq:operatoreq}, is often useful in existing methods.
	To ensure the existence of a solution of \eqref{eq:operatoreq} and obtain an explicit enclosure of the solution numerically,
	it is important to confirm the invertibility of $L$ (or ${\mathcal L}$) and estimate an upper bound of the inverse norm $ \invnormzero $ (or $\| {\mathcal L}^{-1} \|$) \cite{oishi1995numerical,plum1992explicit,plum1994enclosures,plum2001computer,nakao2005numerical, plum2008,nakao2011numerical,watanabe2013posteriori,tanaka2014verified,nakao2015some,nakao2019numerical,watanabe2020some,sekine2020new}.
	The existing methods are categorized as follows:	
	\begin{description}
	 \item[Category 1:] Methods that solve a system of linear elliptic differential equations
	    \begin{itemize}
	      \item $\| {\mathcal L}^{-1} \|_{{\mathcal B}(V^*, V)}$: \cite{oishi1995numerical,nakao2005numerical,nakao2015some}, \cite[Part I]{nakao2019numerical}, \cite{sekine2020new}
	      \item $\| L^{-1} \|_{{\mathcal B}(X, V)}$: \cite{watanabe2013posteriori,watanabe2020some}, \cite[Part I]{nakao2019numerical}, \cite{sekine2020new}
	    \end{itemize}
	 \item[Category 2:] Methods that consider infinite-dimensional eigenvalue problems
	   \begin{itemize}
	      \item $\| {\mathcal L}^{-1} \|_{{\mathcal B}(V^*, V)}$: \cite{plum1994enclosures,tanaka2014verified}, \cite[Part II]{nakao2019numerical}
	      \item $\| L^{-1} \|_{{\mathcal B}(X, V)}$: None so far (the objective of this paper)
	    \end{itemize}
	\end{description}
	Methods in Category 1 estimate the desired norms $ \invnormzero $ and $\| {\mathcal L}^{-1} \|$ by solving the linear problems
	\begin{align*}
		&\mbox{Find} ~ \phi \in V ~~ \mbox{s.t.} ~~ {\mathcal L} \phi = g \in V^*
		\shortintertext{and}
		&\mbox{Find} ~ \phi \in {\mathcal D}( A) ~~ \mbox{s.t.} ~~ L \phi = g \in X,
	\end{align*}
	respectively.
	The third author of this paper has previously proposed a method in Category 1 in a general Banach space setting \cite{oishi1995numerical}.
	Although this method would give a rough estimation of the inverse norm, it is widely applicable to both integral equations and differential equations.	
	Subsequently, several works \cite{nakao2005numerical,nakao2015some,nakao2019numerical,watanabe2013posteriori,watanabe2020some,sekine2020new} further developed the methods in Category 1 and applied them to several problems.
	In particular, Nakao, Hashimoto, and Watanabe \cite{nakao2005numerical} proposed an effective evaluation method for the inverse norm based on constructive a priori error estimates on Hilbert spaces.
	Although the methods \cite{nakao2005numerical,nakao2015some,nakao2019numerical,watanabe2013posteriori,watanabe2020some} requires a technical norm evaluation of $\phi$, which has been essential so far,
	the resent study \cite{sekine2020new} designed a Block-Gaussian-elimination based method that requires no such technical evaluation but directly expresses $\phi$.
	
	Methods in Category 2, different approaches from Category 1, consider the eigenvalue problem corresponding to the focused operator $L$ or $\mathcal L$, trying to estimate its eigenvalues directly.
	When ${\mathcal L}$ is self-adjoint, these methods consider the infinite-dimensional eigenvalue problem 
	\begin{align}
		\label{eq:eig:selfadjoint}
		\mbox{Find} ~ (\lambda, \phi) \in {\mathbb R} \times V ~~ \mbox{s.t.} ~~ {\mathcal L} \phi = \lambda {\mathcal A} \phi
	\end{align}
	and use the equality $\| {\mathcal L}^{-1} \| = | \lambda_1 |^{-1}$,
	where $\lambda_1$ is the eigenvalue, the absolute value of which is minimal in the spectrum. 
	Plum \cite{plum1994enclosures} proposed the first method in Category 2 to estimate $\| {\mathcal L}^{-1} \|$ that is applicable to both bounded and unbounded domains $\Omega$.
	Plum estimated a rough lower bound for $|\lambda_1|$ using a homotopy based method (see, \cite{plum1991bounds} and \cite[Part II]{nakao2019numerical}) and refined its accuracy using the Temple-Lehman-Goerisch method (see \cite{lehmann1963optimale,Goerisch:1990:DGB:120387.120394,plum1994enclosures}).
	Subsequently, Liu and the third author of this paper proposed an a priori error estimation method for evaluating the eigenvalues of the weak Laplacian $ \mathcal {A} $ \cite{liu2013verified}.
	By applying this method to \eqref{eq:eig:selfadjoint},
	$ \| {\mathcal L}^{-1} \| $ was successfully evaluated in \cite{tanaka2014verified}.
	The eigenvalue evaluation method in \cite{liu2013verified} was further extended to a more general setting by Liu \cite{liu2015framework}, which will play an important role in the objective of this paper.
	The accuracy of the eigenvalues can be further improved using the Temple-Lehman-Goerisch method (again, see \cite{lehmann1963optimale,Goerisch:1990:DGB:120387.120394,plum1994enclosures}).
	Therefore, if more accuracy is required, we can implement a refinement step.
	This is a remarkable advantage of eigenvalue based methods,
	whereas comparing Category 1 and 2 in a general level is difficult
	because they have different strengths and weaknesses.
	
	Despite the methods in Category 2 have been applied to $\| {\mathcal L}^{-1} \|$ and helped to solve various problems, they have not succeeded in a direct evaluation of $\| L^{-1} \|$.
	One alternative way is to use the bound $ C_p $ for the embedding $ \Ltwo \hookrightarrow H^{-1}(\Omega) $ (called the Poincar\'e constant) that satisfies
	$\| \phi \|_{H^{-1}} \le C_p \| \phi \|_{L^2}, ~ \phi \in \Ltwo$ and rely on the inequality $\| L^{-1} \| \le C_p \| {\mathcal L}^{-1} \|$.
	However, this may lead to overestimation.
	Therefore, researchers have desired to solve the following questions:
	\begin{description}
	 \setlength{\leftskip}{1.0cm}
	 \item[I. ] Does the eigenvalue problem corresponding to $\| L^{-1} \|_{{\mathcal B}(X, V)}$ exist?
	 \item[II. ] If so, can we solve it?
	\end{description}
	The purpose of this paper is to answer these question.
	We reveal the eigenvalue problem corresponding to the inverse norm $ \| L^{-1} \| $.
	Then, using Liu's evaluation method \cite{liu2015framework},
	we solve the problem to verify the invertibility of $L$ and estimate $ \| L^{-1} \| $ explicitly.
	The accuracy of the estimated eigenvalues is improved using the Temple-Lehman-Goerisch method \cite{lehmann1963optimale,Goerisch:1990:DGB:120387.120394}.
	
	The remainder of this paper is organized as follows.
	In Section \ref{sec:eigpro}, we reveal the eigenvalue problem corresponding to the operator norm $ \| L^{-1} \| $.
	In Section \ref{sec:lowerbound},
	lower bounds for the eigenvalues of the problem are estimated using Liu's method \cite{liu2015framework}.
	In Section \ref{sect:ex}, we present a numerical example where our method is applied to the Lotka-Volterra system with diffusion terms.	
	\section{Eigenvalue problem corresponding to $ \| L^{-1} \|_{{\mc B}(X,V)} $}\label{sec:eigpro}
%
	
	This section reveals an eigenvalue problem corresponding to the operator norm $ \| L^{-1} \|_{{\mc B}(X,V)} $.
	We assume $Q \in \mc{B}(X)$ and denote its dual operator by $ Q^*\in \mc{B}(X) $.
	In our method, the operators $A^{-\frac{1}{2}}: X \ra X$ and $\sqA: {\mc D}(\sqA) \subset X \ra X$ play an important role. They are defined by the spectral decomposition
	\begin{align*}
	A^{-\frac{1}{2}}u := \sum_{i=1}^{\infty} \mu_i^{-\frac{1}{2}} (u, \psi_i)_{X} \psi_i 
	~~~ \mbox{and} ~~~
	A^{\frac{1}{2}} := (A^{-\frac{1}{2}})^{-1},
	\end{align*}
	where $\{ \psi_i \}_{i \in {\mathbb N}}$ is the set of eigenfunctions of $A$ completely orthonormalized in $X$, and $\mu_i$ stands for the eigenvalue corresponding to $\psi_i$.
	Note that the definition of $A^{-\frac{1}{2}}$ is consistent with that given by the Dunford integral \cite[Remark 2.7]{yagi2009abstract}.
	The followings properties of $ \sqA $ are known \cite{yagi2009abstract}:
	\begin{enumerate}
		\item[- ] $ {\mc D}(\sqA) = V $;
		\item[- ] $\sqA$ is positive and self-adjpoint;
		\item[- ] $\sqA$ has a compact inverse $\sqAinv: X \ra X$;
		\item[- ] $(\sqA u, \sqA v)_{X} = (u, v)_{V}, ~ u,v \in V$;
		\item[- ] $(\sqAinv u, \sqAinv v)_{V} = (u, v)_{X},~~u,v \in X$.
	\end{enumerate}
	
	We define the densely-defined closed linear operator $T: V \subset X \ra X$ by
	\begin{align*}
	T := \sqA - \sqAinv  Q^*,
	\end{align*}
	obtaining the relation
		$(T u, T v)_{X} = ( u,  v)_{V} - ( (Q + Q^*) u, v)_{X} + (A^{-1} Q^* u, Q^* v)_{X}$.
	The following eigenvalue problem corresponds to the desired norm $ \| L^{-1} \|_{{\mc B}(X,V)} $ (this is proved in Theorem \ref{Thm:eig_inverse_norm1} below):
	\begin{align}
	\label{eq:eig_prob}
	\mbox{Find} ~ (u, \lambda') \in V \times {[0,\infty)} ~~~ \mbox{s.t.}~~~
	(T u, T v)_{X} = \lambda' ( u, v)_{X} ~~\mbox{for~all}~~ v \in V.
	\end{align}

	\begin{Thm}\label{Thm:eig_inverse_norm1}
		Let $\lambda_{\min}' \geq 0$ be the minimal eigenvalue of \eqref{eq:eig_prob}.
		If $\lambda_{\min}' > 0$, then $L$ is bijective and
		\begin{align}
		\| L^{-1} \|_{{\mc B}(X,V)} = \f{1}{ \sqrt{\lambda_{\min}' } }.
		\end{align}
	\end{Thm}
	
	\begin{proof}
		We consider the minimal positive constant $K_T$ that satisfies
		\begin{align*}
		 \| \phi \|_{X} \le K_T \| T \phi \|_{X}, ~ \phi \in V.
		 \end{align*}
		A variational method ensures that
		\begin{align*}
		\f{1}{K_T^2}
		= \inf_{\phi \in V \setminus \{0\}} \f{ \| T \phi \|_{X}^2 }{\| \phi \|_{X}^2 }
		= \inf_{\phi \in V \setminus \{0\}} \f{ ( T \phi , T \phi )_{X} }{ ( \phi, \phi)_{X} }
		= \lambda_{\min}'.
		\end{align*}
		Therefore, when $\lambda_{\min}'>0$, we have
		\[
			K_T = \frac{1}{\sqrt{\lambda_{\min}'} }~(<\infty).
		\]
		This ensures that $T$ is injective because
		\begin{align*}
			T \phi = 0 \Leftrightarrow \| T \phi \|_{X} = 0 \Rightarrow \| \phi \|_{X} = 0 \Leftrightarrow \phi = 0.
		\end{align*}
		Thus, we have $\mbox{Null}(T) := \{ u \in {\mathcal D}(T) ~:~ Tu = 0 \} = \{ 0 \}$.
		
		Because $\| A^{\frac{1}{2}} u \|_{X} = \| u \|_{V}$ for all $u \in V = {\mathcal D}(A^{\frac{1}{2}})$,
		we define an operator $\tilde{A}^{\frac{1}{2}} \in {\mathcal B}(V, X)$ with the relation $A^{\frac{1}{2}} u = \tilde{A}^{\frac{1}{2}} u$ for all $u \in V$.
		We denote $\tilde{T} := \tilde{A}^{\frac{1}{2}} - A^{-\frac{1}{2}}Q^{*}|_{V} \in {\mathcal B}(V, X)$, where $\cdot|_V$ stands for the restriction of the domain.
		Because $T u = \tilde{T} u$ for all $u \in V = {\mathcal D}(T)$, we have $\mbox{Null}(T) = \mbox{Null}(\tilde{T})$ and ${\mathcal R}(T) = {\mathcal R}(\tilde{T})$, where ${\mathcal R}(T)$ stands for the range of $T$.
		Because $\tilde{T}$ is injective and $\tilde{A}^{-\frac{1}{2}}$ is bijective, $\tilde{A}^{-\frac{1}{2}} \tilde{T}$ is injective.
		Also, $\tilde{A}^{-\frac{1}{2}} A^{-\frac{1}{2}} Q^*|_V \in {\mathcal B}(V)$ is compact because $A^{-\frac{1}{2}} \in {\mathcal B}(X)$ is compact, and $\tilde{A}^{-\frac{1}{2}} \in {\mathcal B}(X, V)$ and $Q^*|_V  \in {\mathcal B}(V, X)$ are bounded.
		We can express $\tilde{A}^{-\frac{1}{2}} \tilde{T}$ as
		\begin{align*}
			&\tilde{A}^{-\frac{1}{2}} \tilde{T} = I - \tilde{A}^{-\frac{1}{2}} A^{-\frac{1}{2}} Q^*|_V \in {\mathcal B}(V).
		\end{align*}
		Therefore,  the Riesz-Schauder theorem ensures that $\tilde{A}^{-\frac{1}{2}} \tilde{T}$ is bijective
		and so is $\tilde{T}$ because of the bijectivity of $\tilde{A}^{-\frac{1}{2}}$.
		As a result, we ensure that ${\mathcal R}(T) = {\mathcal R}(\tilde{T}) = X$ and $T$ is bijective.
		Thus, $T$ has a bounded inverse, that is, ${\mathcal D}(T^{-1}) = X$ and $T^{-1} \in {\mc B}(X)$.

		Since $T$ is bijective, its adjoint operator $T^* = A^{\frac{1}{2}} - Q A^{-\frac{1}{2}}$ is also bijective.
		Because we have
		\begin{align}
		\label{Ts}
		T^* u 
		= ( \sqA - Q \sqAinv  )u
		= L \sqAinv u, ~~ u \in V,
		\end{align}
		$ L $ is expressed as 
		\begin{align*}
			L u = T^*\sqA u, ~~ u \in {\mathcal D}(A).
		\end{align*}
		Therefore, the bijectivity of $\sqA$ and $T^*$ ensures that of $ L $.
		The norm of $L^{-1}$ is evaluated as
		\begin{align*}
		\| L^{-1} \|_{{\mc B}(X,V)}
		=& \sup_{\psi \in X \setminus \{0\}} \f{ \| L^{-1} \psi \|_{V} }{ \| \psi \|_{X} }
		= \sup_{\psi \in X \setminus \{0\}} \f{ \| A^{\f{1}{2}} L^{-1} \psi \|_{X} }{ \| \psi \|_{X} } 
		= \sup_{\psi \in {\mathcal D}(A) \setminus \{0\}} \f{ \| A^{\f{1}{2}} \psi \|_{X} }{ \| L \psi \|_{X} } \\
		=& \sup_{\psi \in V \setminus \{0\}} \f{ \| \psi \|_{X} }{ \| L A^{-\f{1}{2}} \psi \|_{X} }
		= \sup_{\psi \in V \setminus \{0\}} \f{ \| \psi \|_{X} }{ \| T^* \psi \|_{X} }
		= \sup_{\psi \in X \setminus \{0\}} \f{ \| (T^*)^{-1} \psi \|_{X} }{ \| \psi \|_{X} } = \| (T^*)^{-1} \|_{ {\mc B}(X)}.
		\end{align*}
		Because both $T^{-1}$ and $(T^*)^{-1}$ are bounded linear operators on $X$, we have $(T^*)^{-1} = (T^{-1})^*$ and $\|(T^{-1})^*\|_{ {\mc B}(X)} = \|T^{-1}\|_{ {\mc B}(X)}$.
		Therefore,
		\begin{align*}
			\| L^{-1} \|_{{\mc B}(X,V)}
			=&  \| (T^*)^{-1} \|_{ {\mc B}(X)}
			= \| T^{-1} \|_{ {\mc B}(X)}
			= \sup_{\psi \in X \setminus \{0\}} \f{ \| T^{-1} \psi \|_{X} }{\| \psi \|_{X}}
			= \sup_{\phi \in V \setminus \{0\}} \f{ \| \phi \|_{X} }{\| T \phi \|_{X}} = K_T.
		\end{align*}
		
	\end{proof}

\begin{rem}
	An alternative problem for \eqref{eq:eig_prob} corresponding to $\| L^{-1} \| $ is
	\begin{align}
		\mbox{Find} ~~ (u, \lambda') \in V \times [0,\infty) ~~~ \mbox{s.t.} ~~~ (  T^*  u,  T^* v)_{X}  = \lambda' ( u, v)_{X} ~~\mbox{for~all}~~ v \in V,
		\label{eq:eig_prob2}
	\end{align}
	where $(T^* u, T^* v)_{V} = ( u, v )_{V} - ( A^{\frac{1}{2}} u, Q A^{-\frac{1}{2}} v )_{X} - ( Q A^{-\frac{1}{2}} u, A^{\frac{1}{2}} v )_{X} + ( Q A^{-\frac{1}{2}} u, Q A^{-\frac{1}{2}} v )_{X}$.
	If the minimal eigenvalue $\lambda_{\min}'$ of \eqref{eq:eig_prob2} is positive, we have $\| L^{-1} \| = 1/\sqrt{\lambda_{\min}'}$
	(see Theorem \ref{Thm:eig_inverse_norm2}).
	However, it is problematic to compute the terms associated with $A^{\frac{1}{2}}$ or $A^{-\frac{1}{2}}$ directly.
	Another alternative problem equivalent to \eqref{eq:eig_prob2} is
	\begin{align}
	\mbox{Find} ~~ (u, \lambda') \in {\mc D}(A) \times [0,\infty) ~~~ \mbox{s.t.} ~~~ (  L  u,  L v)_{X}  = \lambda' ( u, v)_{V} ~~\mbox{for~all}~~ v \in {\mc D}(A). \label{eq:eig_probL}
	\end{align}
	We obtain this problem by applying the replacements $u'=A^{-\frac{1}{2}}u$ and $v'=A^{-\frac{1}{2}}v$ to \eqref{eq:eig_prob2},
	eliminating the terms including $A^{\frac{1}{2}}$ or $A^{-\frac{1}{2}}$.
	Problem \eqref{eq:eig_probL} includes the term $(Au, Av)_X$ and requires higher-order approximations to estimate the desired eigenvalue.
	We can avoid such problematic calculations by considering \eqref{eq:eig_prob}.
\end{rem}

	\section{Lower bound of minimal eigenvalue}\label{sec:lowerbound}
	In this section, we estimate a lower bound for the minimal eigenvalue $\lambda_{\min}'$ of \eqref{eq:eig_prob} using Liu's method \cite{liu2015framework}.
	For this purpose, we introduce an infinite-dimensional eigenvalue problem as follows.
	Let $R_h: V \ra V_h$ be the Ritz projection defined by
	\begin{align}
	((I - R_h) u, v_h)_{V} = 0, ~ u \in V, ~  v_h \in V_h,
	\label{def:Ritz-projection}
	\end{align}
	where $V_h \subset V$ is a finite-dimensional subspace spanned by the bases $\{ \phi_1, ~ \phi_2, \cdots,~ \phi_N \}$.
	Replacing the term $(A^{-1} Q^* u, Q^* v)_{X}$ in \eqref{eq:eig_prob} to $(R_h A^{-1} Q^* u, Q^* v)_{X}$,
	we obtain the eigenvalue problem
	\begin{align}
		\mbox{Find} ~ (u, \lambda) \in V \times {\mathbb R} ~~~ \mbox{s.t.} ~~~ ( u, v)_{V} - ( (Q + Q^*) u, v)_{X} + (R_h A^{-1} Q^* u, Q^* v)_{X} = \lambda ( u, v)_{X},~v\in V.
		\label{eq:eig_prob_app}
	\end{align}
	All the eigenvalues of \eqref{eq:eig_prob_app} are real.
	This is confirmed by the following consideration.
	Let $P_h : X \ra V_h$ be the orthogonal projection defined by $ ( (I - P_{h}) u, v_h )_{X} = 0, u \in X, v_h \in V_h $.
	We define $A_h : V_h \ra V_h$ by $ (A_h u_h, v_h)_{X} = ( u_h, v_h)_{V}, ~ u_h, v_h \in V_h$.
	Because $R_h A^{-1} = A_h^{-1} P_h$ ( see \cite[(1.24)]{fujita2001operator}), we have
	\begin{align*}
		&(R_h A^{-1} Q^{*} u, Q^{*} v)_{X} = (A_h^{-1} P_h Q^{*} u, Q^{*} v)_{X} \\
		=& (P_h A_h^{-1} P_h Q^{*} u, Q^{*} v)_{X} = (Q P_h A_h^{-1} P_h Q^{*} u,  v)_{X}, ~ u,v \in V.
	\end{align*}
	The reality follows from the relationship
	\begin{align*}
		(Q P_h A_h^{-1} P_h Q^{*} u,  v)_{X} =(u, Q P_h A_h^{-1} P_h Q^{*} v)_{X}, ~ u,v \in V.
	\end{align*}

	For the modified problem \eqref{eq:eig_prob_app}, we prepare the following lemma.
	
	\begin{lem}\label{lem:transform_infinit_eigenvalue}
		The minimal eigenvalues $\lambda_{\min}'$ and $\lambda_{\min}$ respectively of \eqref{eq:eig_prob} and \eqref{eq:eig_prob_app} satisfy the inequality
		\begin{align*}
		\lambda_{\min} \le \lambda_{\min}'.
		\end{align*}
	\end{lem}
	
	\begin{proof}
		For any $u \in V$, we have
		\begin{align*}		
			(A^{-1} u , u )_{X} = (R_h A^{-1} u , u )_{X} + (( I - R_h ) A^{-1} u , ( I - R_h ) A^{-1} u)_{V} \ge (R_h A^{-1} u , u )_{X}.
		\end{align*}
		Therefore, for the eigenfunction $\psi' \in V$ corresponding to $\lambda_{\min}'$,
		\begin{align*}
		\lambda_{\min}' &= \f{ ( \psi' , \psi' )_{V} - ( (Q + Q^*) \psi' , \psi' )_{X} + (A^{-1} Q^* \psi' , Q^* \psi' )_{X} }{ ( \psi' , \psi' )_{X} } \\
		&\ge \inf_{\psi \in V \setminus\{ 0 \}} \f{ (\psi , \psi )_{V} - ( (Q + Q^*) \psi , \psi )_{X} + (R_h A^{-1} Q^* \psi , Q^* \psi )_{X}}{ ( \psi , \psi )_{X} } = \lambda_{\min}.
		\end{align*}
	\end{proof}

	Lemma \ref{lem:transform_infinit_eigenvalue} allows us to use $\lambda_{\min}$ as a lower bound for $\lambda_{\min}'$.
	To estimate $\lambda_{\min}$ explicitly, we discretize \eqref{eq:eig_prob_app}, obtaining the following problem
	\begin{align}
	&\mbox{Find} ~ (u_h, \lambda_h) \in V_h \times {\mathbb R} ~~~ \mbox{s.t.}~~~ \notag \\
	& ~~ ( u_h,  v_h)_{V} - ( (Q + Q^*) u_h, v_h)_{X} + (R_h A^{-1} Q^* u_h, Q^* v_h)_{X} = \lambda_h ( u_h, v_h)_{X},~ v_h \in V_h.
	\label{eq:finite_eig_prob_compute}
	\end{align}
	This is equivalent to the matrix eigenvalue problem
	\begin{align}
	\mbox{Find} ~ (\vector{x}, \lambda_h)^T \in {\mathbb R}^N \times {\mathbb R} ~ \mbox{s.t.} ~ \left( \vector{D} - \left( \vector{Q} + \vector{Q}^T \right) + \vector{Q} \vector{D}^{-1} \vector{Q}^T \right) \vector{x} = \lambda_h \vector{L} \vector{x}, 
	\label{eq:finite_eig_prob_compute_matrix}
	\end{align}
	where $\vector{D}, ~ \vector{L}, ~ \vector{Q} \in {\mathbb R}^{N \times N}$ are real matrices respectively defined by
	\[
	\vector{D}_{ij} := ( \phi_j, \phi_i)_{V}, ~
	\vector{L}_{ij} := ( \phi_j, \phi_i)_{X}, ~
	\vector{Q}_{ij} := ( Q \phi_j, \phi_i)_{X},\]
	and $\vector{Q}^T$ denotes the transposed matrix of $\vector{Q}$.

	By applying Liu's method \cite{liu2015framework} to \eqref{eq:eig_prob_app} and \eqref{eq:finite_eig_prob_compute_matrix}, we estimate eigenvalues $\lambda$ of \eqref{eq:eig_prob_app} in the following theorem.
	Hereafter, $C_p$ denotes the Poincar\'e constant that satisfies $ \| u \|_{X} \le C_p \| u \|_{V}$ for all $ u \in V$.
	A positive number $\sigma$ is chosen to satisfy $ \sigma > \| Q + Q^* \|_{{\mc B}(X,X)} $.
	
	\begin{Thm}\label{Thm:Computed_Theorem}
		Let $\lambda^{(i)}$ and  $\lambda_h^{(i)}~(i=1, 2, \cdots)$ be the $i$th eigenvalues of \eqref{eq:eig_prob_app} and \eqref{eq:finite_eig_prob_compute_matrix}, respectively.
		Suppose that there exists a positive constant $C_h$ that satisfies
		\begin{align*}
			\| (I - R_h) u \|_{X} \le C_h \| (I - R_h) u \|_{V}
		\end{align*}
		for all $u \in V$.
		Let us define $C_{M_{\sigma}}$ as
		\begin{align}
		C_{M_{\sigma}} := C_h \left( 1 + \frac{ \|  \sigma  - \left( Q + Q^* \right) \|_{{\mc B}(X)} + C_p^2 \| Q \|_{{\mc B}(X)}^2 }{ \lambda_h^{(1)} + \sigma } \right).
		\label{def:CMs}
		\end{align}
		Then, we have
		\begin{align}
		\f{ \lambda_h^{(i)} + \sigma }{1 + C_{M_{\sigma}}^2 \left( \lambda_h^{(i)} + \sigma \right) } - \sigma
		\le
		\lambda^{(i)}
		\le
		\lambda_h^{(i)}.
		\label{inq:eigenvalue_inclusion}
		\end{align}
	\end{Thm}
	
	\begin{rem}\label{rem:Linv_estimate}
		Although Theorem \ref{Thm:Computed_Theorem} gives a rough lower bound, particularly when $ \sigma $ is large,
		the lower bound approaches $ \lambda^{(i)} $ as $ C_h \downarrow 0$.
		If the lower bound in \eqref{inq:eigenvalue_inclusion} for $ i=1 $ is positive, Lemma \ref{lem:transform_infinit_eigenvalue} guarantees $\lambda'_{\min}>0$ and thus $L$ is nonsingular,
		the norm of which is evaluated as
		\begin{align*}
		\| L^{-1} \|_{{\mc B}(X,V)} \le  \left(\f{ \lambda_h^{(1)} + \sigma }{1 + C_{M_{\sigma}}^2 \left( \lambda_h^{(1)} + \sigma \right) } - \sigma \right)^{-\f{1}{2}}.
		\end{align*}
		When the positivity of $\lambda'_{\min}$ is not confirmed via Theorem \ref{Thm:Computed_Theorem} or more accuracy is required, we try to improve the accuracy using the Temple-Lehman-Goerisch method {\rm \cite{Goerisch:1990:DGB:120387.120394}}.
	\end{rem}

	\subsection{Proof of Theorem \ref{Thm:Computed_Theorem}, Step1 --- Applying Liu's theorem to \eqref{eq:eig_prob_app}}\label{subsec:Liu_Theorem}
	Because the right-hand side inequality in \eqref{inq:eigenvalue_inclusion} immediately follows from the Rayleigh-Ritz bound, we prove the left-hand side inequality.
	Liu's theorem \cite[Theorem 2.1]{liu2015framework} requires the positive definiteness of the bilinear form in the left-hand side of \eqref{eq:eig_prob_app}.
	Therefore, we add a perturbation so that the minimal eigenvalue of \eqref{eq:eig_prob_app} becomes positive.
	More precisely, we add $\sigma(u, v)_{X}$ to both sides, and rewrite the problem as
	\begin{align}
	\mbox{Find} ~ (u, \mu_{\sigma})^T \in V \times {\mathbb R} ~~~ \mbox{s.t.}~~~
	\minner{u}{v}
	= \mu_{\sigma} (u, v)_{X},~~v \in V,
	\label{eq:eigvalue_mu_sigma}
	\end{align}
	where $\mu_{\sigma} = \lambda + \sigma$ and $ \minner{u}{v} $ is the bilinear form defined by
	\begin{align}
	\minner{u}{v} := ( u, v)_{V} - ( (Q + Q^*) u, v)_{X} + (R_h A^{-1} Q^* u, Q^* v)_{X} + \sigma(u, v)_{X}.
	\label{def:Msigma}
	\end{align}
	We then show that $ \minner{\cdot}{\cdot} $ is an inner product of $ V $ to use \cite[Theorem 2.1]{liu2015framework}.
	Because $ R_h A^{-1} = A_h^{-1} P_h$, we have
	\begin{align}
	(R_h A^{-1} Q^* u, Q^* u)_{X}
	= (A_h^{-1} P_h Q^* u, Q^* u)_{X}
	= ( Q P_h A_h^{-1} P_h Q^* u,  u)_{X}
	= \|A_h^{-\f{1}{2}} P_h Q^* u\|_{X}^2,~ u \in V.
	\label{selfAdofB}
	\end{align}
	This ensures that the third term of $ \minner{u}{v} $ in \eqref{def:Msigma} is nonnegative and symmetric with respect to $ u,v \in V $.
	Moreover, the inequality $ \sigma > \| Q + Q^* \|_{{\mc B}(X,X)} $ leads to
	\begin{align}
	\left( \sigma u, u \right)_{X} - \left( (Q+Q^*) u, u \right)_{X} \ge \left( \sigma - \| Q + Q^* \|_{{\mc B}(X,X)} \right) \| u \|_{X}^2 > 0.
	\label{inq:sigma_positive}
	\end{align}
	Consequently, we confirm that
	\begin{align*}
	\minner{u}{u} > 0, ~ u \in V \setminus \{ 0 \}~~~\mbox{and}~~~
	\minner{u}{v} = \minner{v}{u}, ~ u, v \in V,
	\end{align*}
	that is, the bilinear form $ \minner{\cdot}{\cdot} $ is an inner product, and therefore we can use $ \mnorm{u} := \sqrt{ \minner{u}{u} },~ u\in V$ as a norm for $ V $.
	
	Let us define anther orthogonal projection $R_{M_{\sigma}} : V \ra V_h$ with respect to $ \minner{\cdot}{\cdot} $ by
	\begin{align*}
	\minner{(I - R_{M_{\sigma}}) u}{v_h} = 0, ~ v_h \in V_h.
	\end{align*}
	By applying \cite[Theorem 2.1]{liu2015framework} to \eqref{eq:eigvalue_mu_sigma}, 
	we confirm inequality \eqref{inq:eigenvalue_inclusion} for a 
	positive number $C_{M_{\sigma}}$ that satisfies
	\begin{align}
	\| (I - R_{M_{\sigma}}) u \|_{X} \le C_{M_{\sigma}} \| (I - R_{M_{\sigma}}) u \|_{M_{\sigma}}, ~ u \in V.
	\label{inq:CMs}
	\end{align}
	\subsection{Proof of Theorem \ref{Thm:Computed_Theorem}, Step2 --- Calculation of $C_{M_{\sigma}}$}\label{subsec:Calc_CMs}
	To complete the proof, 
	we confirm that $C_{M_{\sigma}}$ defined by \eqref{def:CMs} satisfies \eqref{inq:CMs}.
	In the following proof is inspired by the approach provided in \cite[Theorem 5]{nakao2008guaranteed}.
	Let us define $B:X \ra X$ by
	\begin{align*}	
		B := \sigma I - ( Q + Q^*) + Q^* R_h A^{-1} Q^*.
	\end{align*}	
	Because \eqref{selfAdofB} and \eqref{inq:sigma_positive} ensures $(B u, u)_{X} > 0, ~ u \in V$,
	we have
	\begin{align*}
		\| u \|_{V}^2 = (u, u)_{V} \le (u, u)_{V} + \left( B u, u \right)_{X} = \| u \|_{M_\sigma}^2, ~ \forall u \in V.
	\end{align*}
	Moreover, for any $u \in V$,
	\begin{align}
		\| (I - R_{M_{\sigma}}) u \|_{X} \le& \| (I - R_h) u \|_{X} + \| (R_h - R_{M_{\sigma}}) u \|_{X} \notag \\
		\le& C_h \| (I - R_h) u \|_{V} + \| R_{M_{\sigma}} ( I - R_h) u \|_{X}. \label{inq:wait}
	\end{align}
	We set $e := (I-R_h)u$. 
	For any $v_h \in V_h$,
	\begin{align*}
		\minner{R_{M_{\sigma}} e}{v_h} =& \minner{e}{v_h} = ( e, v_h)_{V} + ( B e, v_h)_{X} \\
		=& ((I - R_h)u, v_h)_{V} + ( B e, v_h)_{X} = ( B e, v_h)_{X} \\
		=&  ( P_h B e, v_h)_{X} =  ( P_h \psi, v_h)_{X},
	\end{align*}
	where $\psi := B e$.
	Let $\vector{G} :=  \vector{D} + \sigma \vector{L} - \left( \vector{Q} + \vector{Q}^T \right) + \vector{Q} \vector{D}^{-1} \vector{Q}^T  \in {\mathbb R}^{N \times N}$.
	This implies that
	\begin{align*}
		\vector{G} \vector{e} = \vector{L} \vector{\psi},
	\end{align*}
	where $\vector{e}$ and $\vector{\psi}$ are coefficient vectors of $R_{M_{\sigma}} e$ and $P_h \psi$, respectively.
	We denote the Euclidean norm and its matrix norm by $\| \cdot \|_{E}$.
	Since $\vector{L}$ is a symmetric positive-definite matrix, it can be decomposed to $\vector{L} = \vector{L}^{\frac{T}{2}} \vector{L}^{\frac{1}{2}}$, where $\cdot^T$ indicates transposition, and $\cdot^{\frac{T}{2}} = (\cdot^{\frac{1}{2}})^{T}$.
	Therefore, we have
	\begin{align}
		\| R_{M_{\sigma}} ( I - R_h) u \|_{X} =& \| R_{M_{\sigma}} e \|_{X} = \| \vector{L}^{\frac{1}{2}} \vector{e} \|_{E} = \| \vector{L}^{\frac{1}{2}} \vector{G}^{-1} \vector{L}^{\frac{T}{2}} \vector{L}^{\frac{1}{2}} \vector{\psi} \|_{E} \notag \\
		\le& \| \vector{L}^{\frac{1}{2}} \vector{G}^{-1} \vector{L}^{\frac{T}{2}} \|_{E} \| \vector{L}^{\frac{1}{2}} \vector{\psi} \|_{E} = \| \vector{L}^{\frac{1}{2}} \vector{G}^{-1} \vector{L}^{\frac{T}{2}} \|_{E} \| P_h \psi \|_{X} \notag \\
		 \le& \| \vector{L}^{\frac{1}{2}} \vector{G}^{-1} \vector{L}^{\frac{T}{2}} \|_{E} \| B e \|_{X}. \label{inq:EBwait}
	\end{align}
	Because $(\cdot, \cdot)_{M_\sigma}$ is an inner product, $\vector{G}$ is a symmetric positive-definite matrix.
	Hence,
	\begin{align*}
		\| \vector{L}^{\frac{1}{2}} \vector{G}^{-1} \vector{L}^{\frac{T}{2}} \|_{E} = \frac{1}{ \lambda_h^{(1)} + \sigma }.
	\end{align*}
	We obtain the following estimation
	\begin{align*}
		\| B \|_{{\mc B}(X)} &= \sup_{u \in X} \f{ \| \left( \sigma  - \left( Q + Q^* \right) + Q P_h A_h^{-1} P_h Q^* \right) u \|_{X} }{ \| u \|_{X} }\\
		&\le \|  \sigma  - \left( Q + Q^* \right) \|_{{\mc B}(X)} + \| Q \|_{{\mc B}(X)}  \sup_{u \in X} \f{\| P_h A_h^{-1} P_h Q^* u \|_{X} }{ \| u \|_{X} }.
	\end{align*}
	The second term is further evaluated as
	\begin{align*}
		\sup_{u \in X} \f{ \| P_h A_h^{-1} P_h Q^* u \|_{X} }{ \| u \|_{X} }
		&\le \sup_{u \in X} \f{ \| A_h^{-1} P_h Q^* u \|_{X} }{ \| u \|_{X} }
		\le C_p \sup_{u \in X} \f{ \| A_h^{-1} P_h Q^* u \|_{V} }{ \| u \|_{X} } 
		=C_p  \sup_{u \in X} \f{ \| R_h A^{-1} Q^* u \|_{V} }{ \| u \|_{X} } \\
		&\le C_p \sup_{u \in X} \f{ \| A^{-1} Q^* u \|_{V} }{ \| u \|_{X} } 
		\le C_p^2 \sup_{u \in X} \f{ \| Q^* u \|_{X} }{ \| u \|_{X} } 
		= C_p^2 \| Q^* \|_{{\mc B}(X)}.
	\end{align*}
	It follows from the relation $ \| Q \|_{{\mc B}(X)} =\| Q^* \|_{{\mc B}(X)}  $ that
	\begin{align*}
		\| B \|_{{\mc B}(X)} 
		\le \|  \sigma  - \left( Q + Q^* \right) \|_{{\mc B}(X)}  + C_p^2 \| Q \|_{{\mc B}(X)} ^2.
	\end{align*}
	From \eqref{inq:EBwait}, we obtain
	\begin{align*}
		\| R_{M_{\sigma}} ( I - R_h) u \|_{X} \le& \| \vector{L}^{\frac{1}{2}} \vector{G}^{-1} \vector{L}^{\frac{T}{2}} \|_{E} \| B e \|_{X} \\
		\le& \frac{ \|  \sigma  - \left( Q + Q^* \right) \|_{{\mc B}(X)} + C_p^2 \| Q \|_{{\mc B}(X)}^2 }{ \lambda_h^{(1)} + \sigma }  \| e \|_{X} \\
		\le& C_h \frac{ \|  \sigma  - \left( Q + Q^* \right) \|_{{\mc B}(X)}  + C_p^2 \| Q \|_{{\mc B}(X)}^2 }{ \lambda_h^{(1)} + \sigma }  \| (I - R_h) u \|_{V}.
	\end{align*}
	As a result, it follows from \eqref{inq:wait} that
		\begin{align*}
		\| (I - R_{M_{\sigma}}) u \|_{X} \le& C_h \| (I - R_h) u \|_{V} + \| R_{M_{\sigma}} ( I - R_h) u \|_{X} \\
		\le& C_{M_\sigma} \| (I - R_h) u \|_{V} \\
		\le& C_{M_\sigma} \| (I - R_{M_\sigma}) u \|_{V} \\
		\le& C_{M_\sigma} \| (I - R_{M_\sigma}) u \|_{M_\sigma}
	\end{align*}
	Thus, we complete the proof of Theorem \ref{Thm:Computed_Theorem}.
	\qed
	\section{Numerical Example}\label{sect:ex}
	In this section, we present an example where our method is applied to calculating the inverse norm $ \invnormXV $ derived from the boundary value problem of the Lotka-Volterra system with diffusion terms:
	\begin{align}
	\left\{
	\begin{array}{ll}
	-\Delta{u} = 6u - u^2 + 2uv & {\rm{in}} \hspace{0.1cm} \Omega{,}\\
	-\Delta{v} = 4v + 4uv - v^2 & {\rm{in}} \hspace{0.1cm} \Omega{,}\\
	u = v = 0                                & \partial \Omega{,}
	\end{array}
	\right.
	\label{example_program}
	\end{align}
	on $\Omega = (0,1)^2$.
	For this problem, we set $V=(H^1_0(\Omega))^2$ and $X=(L^2(\Omega))^2$.
	All computations were implemented on a computer with 2.90 GHz Intel Xeon Platinum 8380H CPU $\times$ 4, 3 TB
	RAM, and CentOS 8 using C++11 with GCC version 8.3.1.
	All rounding errors were strictly estimated using the toolboxes kv Library \cite{kashiwagikv} and VCP Library \cite{sekine2021numerical}.
	Therefore, the mathematical correctness of all the results were guaranteed.
	We constructed approximate solutions for \eqref{example_program} from a Legendre polynomial basis.
	We define the set $\{ \psi_1, \psi_2, \cdots \psi_N \}$ of Legendre polynomials by
	\begin{align*}
	\psi_i(x) := \frac{1}{i(i+1)}x(1-x)\frac{dP_i}{dx}(x)~~\mbox{with}~~P_i = \frac{(-1)^i}{i!}\left( \frac{d}{dx} \right)^i x^i (1 - x)^i, ~ i = 1,2, \cdots,
	\end{align*}
	and define the finite dimensional subspace $V_h^N := \left( \mbox{span}( \psi_1, \psi_2, \cdots \psi_N ) \otimes \mbox{span}( \psi_1, \psi_2, \cdots \psi_N ) \right)^2$, where $\otimes$ stands for the tensor product.
	We computed approximate solutions $\hat{u}, \hat{v} \in V_h^{40}$, obtaining the graphs in Figure \ref{Fig:approximate_splution}.

	\begin{figure}[h]
		\begin{minipage}{0.5\hsize}
			\begin{center}
				\includegraphics[width=60mm]{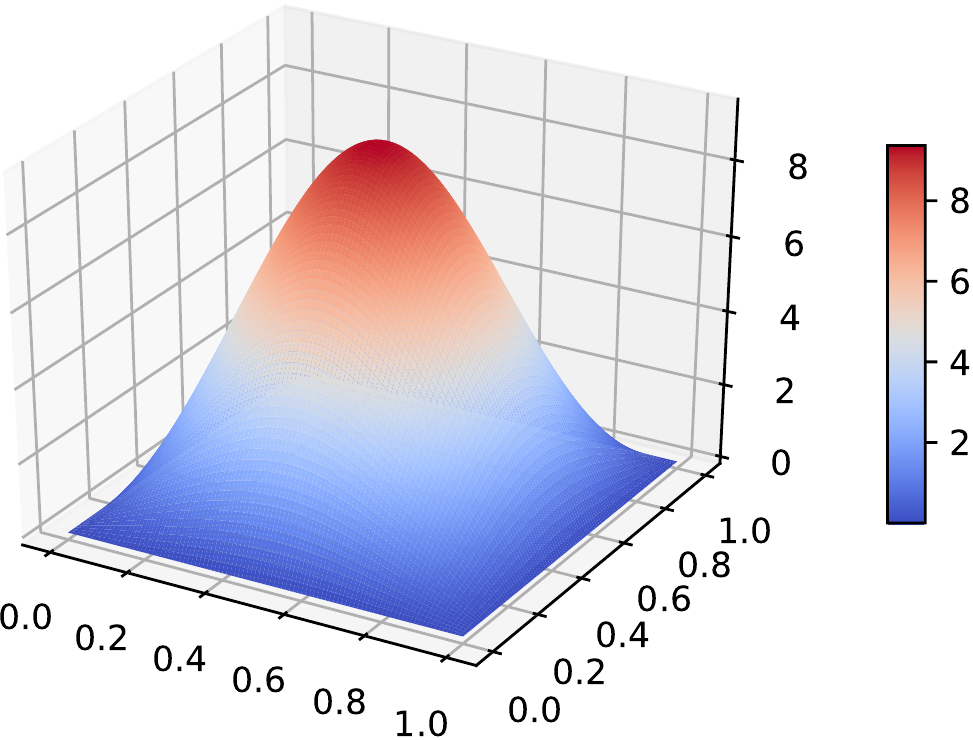}
			\end{center}
		\end{minipage}
		\begin{minipage}{0.5\hsize}
			\begin{center}
				\includegraphics[width=60mm]{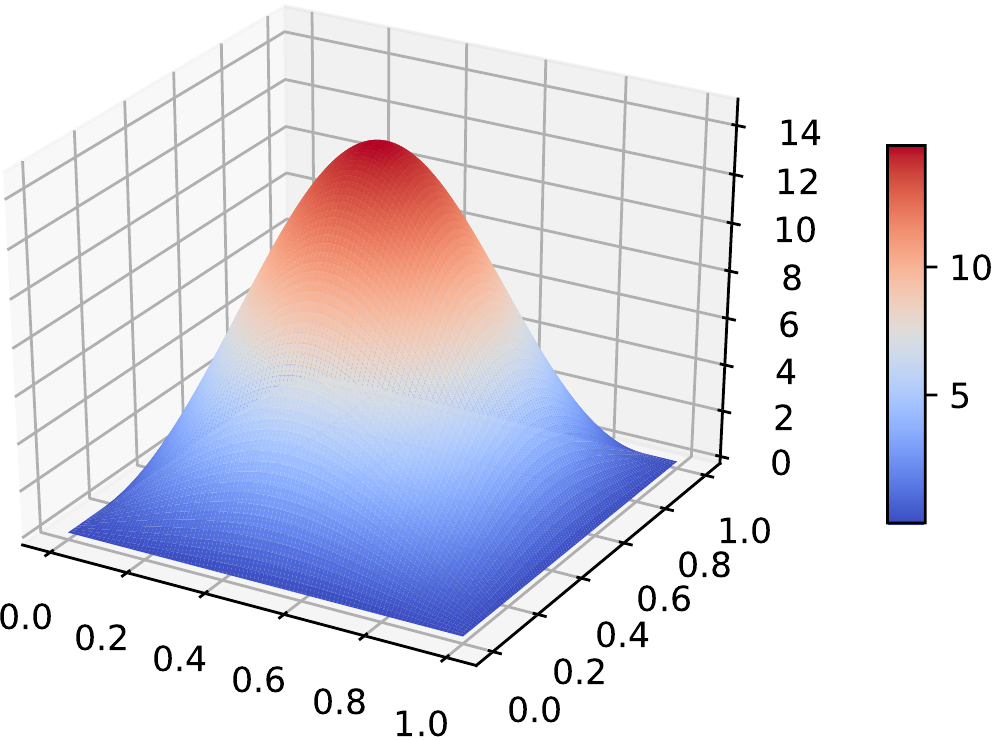}
			\end{center}
		\end{minipage}
		\caption{Approximate solutions of \eqref{example_program}. Left: $\hat{u}$, Right: $\hat{v}$.}
		\label{Fig:approximate_splution}
	\end{figure}

	The linearized operator $L : \mathcal D(A) \subset X \ra X$ corresponding to the Fr\'echet derivative of problem \eqref{example_program} at $\{\hat{u}, \hat{v}\}$ is
	\begin{align}
	\label{linearLB}
	L = A- Q
	\mbox{~~with~~}
	Q=
	\left(
	\begin{array}{cc}
	6 - 2\hat{u} + 2\hat{v}  & 2 \hat{u} \\
	4\hat{v} & 4 + 4 \hat{u} - 2\hat{v}
	\end{array}
	\right),
	\end{align}
	where $Q:X \ni (u,v)^T \mapsto (6 - 2\hat{u} + 2\hat{v}) u + 2 \hat{u} v, 4\hat{v} u + (4 \hat{u} - 2\hat{v}) v)^T \in X$.
	Numerical existence proofs for \eqref{example_program} need an explicit estimation of the inverse norm $ \|L^{-1}\| $.
	Whereas $L$ is non-self-adjoint, Theorems \ref{Thm:eig_inverse_norm1} and \ref{Thm:Computed_Theorem} remain applicable.
	We describe the procedure for obtaining an upper bound for the norm $ \|L^{-1}\| $.
	First, applying Theorem \ref{Thm:Computed_Theorem} to problem \eqref{eq:eig_prob_app}, we look for some $ j $ satisfying
	\begin{align}
	\label{rhocond}
	\mbox{Upper~bound~of~} \lambda^{(j-1)} < \mbox{Lower~bound~of~} \lambda^{(j)},
	\end{align}
	and set $\nu$ to be the upper bound of $\lambda^{(j-1)}$ that is confirmed to be distinct from the smallest eigenvalue $\lambda^{(1)}$.
	Next, we obtain a refined lower bound $\lambda^{(1)}$ using the Temple-Lehman-Goerisch method with $\nu$.
	Note that the accuracy of the lower bound improved by the method depends on neither $ i $ nor the accuracy of $ \nu $ as long as \eqref{rhocond} is satisfied (see \cite{Goerisch:1990:DGB:120387.120394,plum1994enclosures} for details about the Temple-Lehman-Goerisch method).
	Finally, we estimate an upper bound for $ \|L^{-1}\| $ using Theorem \ref{Thm:eig_inverse_norm1}.
	
	Recall that the Poincar\'e constant $ C_p $ required in Theorem \ref{Thm:Computed_Theorem} is calculated as $ C_p = 1/(\sqrt{2} \pi) \approx 0.2251 $ with estimating rounding errors strictly.
	The other constants with respect to $ Q $ were estimated as follows:
	\begin{align*}
	&\| Q \|_{{\mathcal B}(X)} \le 82.580303007092866,~~~~
	\| Q + Q^* \|_{{\mathcal B}(X)} \le 155.29112928765162,\\
	&\| \sigma I - (Q + Q^*) \|_{{\mathcal B}(X)} \le 308.06051565984086,
	\end{align*}
	where we set $ \sigma $ to be the floating point number corresponding to $155.29112928765162+0.0001$. Each decimal number represents the nearest floating point number in the 64 bit double precision corresponding to the input. 
	This calculation was carried out dependent not on $V_h$ as required in Theorem \ref{Thm:Computed_Theorem} but only on $ \hat{u} $ and $ \hat{v} $.
	
	Table \ref{Tab:Numerical_result1} shows numerical results for different dimensions of $V_h\,(=V_h^N)$, where the space $V_h$ required in Theorem \ref{Thm:Computed_Theorem} before using the Temple-Lehman-Goerisch method was chosen as $V_h^N$ constructed from the Legendre polynomials defined at the beginning of this section.
	$\lambda_h^{(1)}$ stands for the inclusion of the minimal eigenvalue of \eqref{eq:finite_eig_prob_compute_matrix}.
	For example, $1.23_{456}^{789}$ denotes the interval $[1.23456, 1.23789]$.
	We obtained negative lower bounds for eigenvalues of \eqref{eq:eig_prob_app} using Theorem \ref{Thm:Computed_Theorem} in all cases.
	Therefore, without the Temple-Lehman-Goerisch method, we failed in showing $\lambda_{\min}' > 0$ required by Theorem \ref{Thm:eig_inverse_norm1}.

	\begin{table}[h]
		\caption{Numerical results using Theorem \ref{Thm:Computed_Theorem} without the Temple-Lehman-Goerisch method.
		}
		\label{Tab:Numerical_result1}
		\begin{center}
			{
				\renewcommand{\arraystretch}{1.2}
				\begin{tabular}{ c | c | c | c | c | }
					\hline
					$N$    & $C_h $                       & $C_{M_{\sigma}}$    &   $ \lambda_h^{(1)}$                                & lower bound of $\lambda^{(1)}$ using \eqref{inq:eigenvalue_inclusion}      \\ \hline \hline
					$ 60 $ & $7.88 \times 10^{-3}$  & $0.0399$               &   $5.861693_{1651409078}^{2446435773}$  & $-26.963087160457065$ \\ \hline
					$ 80 $ & $5.99 \times 10^{-3}$  & $0.0303$               &   $5.861693_{0544573868}^{3569412915}$  & $-14.900551378456357$\\ \hline
					$100 $ & $4.84 \times 10^{-3}$ & $0.0245$               &  $5.86169_{27624274461}^{36483586893}$ & $-8.2725800704491519$ \\ \hline
				\end{tabular}
			}
		\end{center}
	\end{table}

	Table \ref{Tab:Numerical_result} shows the results from Theorem \ref{Thm:Computed_Theorem} together with the Temple-Lehman-Goerisch method.
	When $N \le 60$, our method failed in finding a pair of eigenvalues satisfying \eqref{rhocond} and also estimating $\|L^{-1}\|$.
	However, for $ N = 80,100$, our method succeeded in finding such pairs and obtained highly accurate lower bounds of $\lambda^{(1)}$ displayed in Table \ref{Tab:Numerical_result1}.
	This implies that the Temple-Lehman-Goerisch method has a significant influence on the accuracy of the inverse norm.
	Given approximate eigenfunctions, the method is known to give almost optimal evaluation of eigenvalues when \eqref{rhocond} is satisfied \cite{lehmann1963optimale,Goerisch:1990:DGB:120387.120394,plum1994enclosures,plum1997guaranteed}.

	\begin{table}[h]
		\caption{Numerical results using Theorem \ref{Thm:Computed_Theorem} without Lehman-Goerisch's method.
		}
		\label{Tab:Numerical_result}
		\begin{center}
			{
				\renewcommand{\arraystretch}{1.2}
				\begin{tabular}{ | c | c | c | c | c | }
					\hline
					$N$      & $j$      & $\nu\,(\ge \lambda^{(j-1)})$         & lower bound of $\lambda^{(1)}$        & $ \invnormXV \le$   \\ \hline \hline
					$ 60 $  & $--$   &   $--$                                        &  $--$                                            & $--$                       \\ \hline
					$ 80 $  & $3$     &   $7.8241841395$                       &  $5.8616914678651141$  &  $0.4131$              \\ \hline
					$100 $ & $3$    &   $7.8240844856$                        &  $5.8616913204530476$  &  $0.4131$              \\ \hline
				\end{tabular}
			}
		\end{center}
	\end{table}

	\section{Conclusion}
	We provided the answers to the following questions:
	\begin{description}
	\setlength{\leftskip}{1.0cm}
	 \item[I. ] Does the eigenvalue problem corresponding to $\| L^{-1} \|_{{\mathcal B}(X, V)}$ exist?
	 \item[II. ] If so, can we solve it?
	\end{description}
	Theorems \ref{Thm:eig_inverse_norm1} and \ref{Thm:eig_inverse_norm2} answered to Question I,
	revealing the eigenvalue problem corresponding to $\| L^{-1} \|$.
	The fractional operator $A^{\frac{1}{2}}$ played an important role in these theorems.
	Using Liu's theorem \cite[Theorem 2.1]{liu2015framework}, we proved Theorem \ref{Thm:Computed_Theorem} that gives a rough lower bound for the eigenvalues of the revealed problem.
	We applied our method to the Dirichlet boundary value problem of the Lotka-Volterra system with diffusion terms.
	We saw that the Temple-Lehman-Goerisch method has a significant influence on the accuracy of $\| L^{-1} \|$.
	\appendix
	\renewcommand{\theThm}{\Alph{section}.\arabic{Thm}}
	\section{Simple eigenvalue problem corresponding to $ \| L^{-1} \| $}\label{sec:another_transform}
	\begin{Thm}\label{Thm:eig_inverse_norm2}
		Let $\lambda_{\min}'$ be the minimal eigenvalue of the problem \eqref{eq:eig_prob2}.
		If $\lambda_{\min}' > 0$, then $L$ is bijective and 
		\begin{align*}
			\| L^{-1} \|_{{\mc B}(X,V)} = \frac{1}{\sqrt{\lambda_{\min}' }}.
		\end{align*}
	\end{Thm}
	
	\begin{proof}
		Consider the minimal positive constant $K$ that satisfies
		$ \| \phi \|_{X} \le K \| L \sqAinv  \phi \|_{X},~\phi \in V. $
		The relationship between $K$ and $\lambda_{\min}'$ is 
		\begin{align*}
		\f{1}{K^2} = \inf_{\phi \in V \setminus \{ 0 \}} \f{ \| L \sqAinv  \phi \|_{X}^2 }{ \| \phi \|_{X}^2 }
		= \inf_{\phi \in V \setminus \{ 0 \} } \f{ ( L \sqAinv  \phi, L \sqAinv  \phi )_{X} }{ ( \phi, \phi )_{X} }
		= \lambda_{\min}'~(> 0).
		\end{align*}
		Therefore, when $\lambda_{\min}' >0$, we have
		$ K = 1/{\sqrt{\lambda_{\min}'}}~(<\infty) $,
		which ensures that $L \sqAinv $ is injective.
		As in the proof of Theorem \ref{Thm:eig_inverse_norm1}, linear operators $L \sqAinv $ and $L$ are bijective.
		Thus, we have
		\begin{align*}
		\| L^{-1} \|_{ {\mc B}(X, V) }
		=& \sup_{ \phi \in X \setminus \{ 0 \} } \f{ \| L^{-1} \phi \|_{V} }{ \| \phi \|_{X} } 
		= \sup_{ \phi \in X \setminus \{ 0 \} } \f{ \| A^{\f{1}{2}} L^{-1} \phi \|_{X} }{ \| \phi \|_{X} } \\
		=& \sup_{ \phi \in X \setminus \{ 0 \} } \f{ \| (L \sqAinv )^{-1} \phi \|_{X} }{ \| \phi \|_{X} } 
		= \inf_{ \phi \in V \setminus \{ 0 \} } \f{ \| L \sqAinv  \phi \|_{X} }{ \| \phi \|_{X} } = K.
		\end{align*}
	\end{proof}	
	\section*{Acknowledgement}
	This work is supported by JST CREST Grant Number JPMJCR14D4, and MEXT as ``Exploratory Issue on Post-K computer'' (Development of verified numerical computations and super high-performance computing environment for extreme researches).
	The first author (K.S.) is supported by JSPS KAKENHI Grant Number 16K17651.
	The second author (K.T.) is supported by JSPS KAKENHI Grant Number JP17H07188 and JP19K14601, and Mizuho Foundation for the Promotion of Sciences.
	We thank the editors and reviewers for giving useful comments to improve the contents of this manuscript.

\bibliographystyle{amsplain} 
\bibliography{ref.bib}

\providecommand{\bysame}{\leavevmode\hbox to3em{\hrulefill}\thinspace}
\providecommand{\MR}{\relax\ifhmode\unskip\space\fi MR }
\providecommand{\MRhref}[2]{%
  \href{http://www.ams.org/mathscinet-getitem?mr=#1}{#2}
}
\providecommand{\href}[2]{#2}
\begin{thebibliography}{10}

\bibitem{fujita2001operator}
Hiroshi Fujita, Norikazu Saito, and Takashi Suzuki, \emph{Operator theory and
  numerical methods}, vol.~30, Elsevier, 2001.

\bibitem{Goerisch:1990:DGB:120387.120394}
Friedrich Goerisch and Zhiqing He, \emph{Computer arithmetic and
  self-validating numerical methods}, Academic Press Professional, Inc., San
  Diego, CA, USA, 1990, pp.~137--153.

\bibitem{kashiwagikv}
Masahide Kashiwagi, \emph{{kv library}}, 2018, \url{http://verifiedby.me/kv/}.

\bibitem{lehmann1963optimale}
N.~J Lehmann, \emph{Optimale eigenwerteinschlieungen}, Numerische Mathematik
  \textbf{5} (1963), no.~1, 246--272.

\bibitem{liu2015framework}
Xuefeng Liu, \emph{A framework of verified eigenvalue bounds for self-adjoint
  differential operators}, Applied Mathematics and Computation \textbf{267}
  (2015), 341--355.

\bibitem{liu2013verified}
Xuefeng Liu and Shin'ichi Oishi, \emph{Verified eigenvalue evaluation for the
  laplacian over polygonal domains of arbitrary shape}, SIAM Journal on
  Numerical Analysis \textbf{51} (2013), no.~3, 1634--1654.

\bibitem{nakao2008guaranteed}
Mitsuhiro~T Nakao and Kouji Hashimoto, \emph{Guaranteed error bounds for finite
  element approximations of noncoercive elliptic problems and their
  applications}, Journal of Computational and Applied Mathematics \textbf{218}
  (2008), no.~1, 106--115.

\bibitem{nakao2005numerical}
Mitsuhiro~T Nakao, Kouji Hashimoto, and Yoshitaka Watanabe, \emph{A numerical
  method to verify the invertibility of linear elliptic operators with
  applications to nonlinear problems}, Computing \textbf{75} (2005), no.~1,
  1--14.

\bibitem{nakao2019numerical}
Mitsuhiro~T Nakao, Michael Plum, and Yoshitaka Watanabe, \emph{Numerical
  verification methods and computer-assisted proofs for partial differential
  equations}, Springer, 2019.

\bibitem{nakao2011numerical}
Mitsuhiro~T Nakao and Yoshitaka Watanabe, \emph{Numerical verification methods
  for solutions of semilinear elliptic boundary value problems}, Nonlinear
  Theory and Its Applications, IEICE \textbf{2} (2011), no.~1, 2--31.

\bibitem{nakao2015some}
Mitsuhiro~T Nakao, Yoshitaka Watanabe, Takehiko Kinoshita, Takuma Kimura, and
  Nobito Yamamoto, \emph{Some considerations of the invertibility verifications
  for linear elliptic operators}, Japan Journal of Industrial and Applied
  Mathematics \textbf{32} (2015), no.~1, 19--31.

\bibitem{oishi1995numerical}
Shin'ichi Oishi, \emph{Numerical verification of existence and inclusion of
  solutions for nonlinear operator equations}, Journal of computational and
  applied mathematics \textbf{60} (1995), no.~1, 171--185.

\bibitem{plum1991bounds}
Michael Plum, \emph{Bounds for eigenvalues of second-order elliptic
  differential operators}, Zeitschrift f{\"u}r angewandte Mathematik und Physik
  ZAMP \textbf{42} (1991), no.~6, 848--863.

\bibitem{plum1992explicit}
\bysame, \emph{Explicit {$H^{2}$}-estimates and pointwise bounds for solutions
  of second-order elliptic boundary value problems}, Journal of Mathematical
  Analysis and Applications \textbf{165} (1992), no.~1, 36--61.

\bibitem{plum1994enclosures}
\bysame, \emph{Enclosures for weak solutions of nonlinear elliptic boundary
  value problems}, Inequalities and applications, World Scientific, 1994,
  pp.~505--521.

\bibitem{plum1997guaranteed}
\bysame, \emph{Guaranteed numerical bounds for eigenvalues}, Lecture notes in
  pure and applied mathematics (1997), 313--332.

\bibitem{plum2001computer}
\bysame, \emph{Computer-assisted enclosure methods for elliptic differential
  equations}, Linear Algebra and its Applications \textbf{324} (2001), no.~1,
  147--187.

\bibitem{plum2008}
\bysame, \emph{Existence and multiplicity proofs for semilinear elliptic
  boundary value problems by computer assistance}, Jahresbericht der Deutschen
  Mathematiker Vereinigung \textbf{110} (2008), no.~1, 19--54.

\bibitem{sekine2021numerical}
Kouta Sekine, Mitsuhiro~T. Nakao, and Shin'ichi Oishi, \emph{Numerical
  verification methods for a system of elliptic pdes, and their software
  library}, Nonlinear Theory and Its Applications, IEICE \textbf{12} (2021),
  no.~1, 41--74.

\bibitem{sekine2020new}
Kouta Sekine, Mitsuhiro~T Nakao, and Shin’ichi Oishi, \emph{A new formulation
  using the schur complement for the numerical existence proof of solutions to
  elliptic problems: without direct estimation for an inverse of the linearized
  operator}, Numerische Mathematik \textbf{146} (2020), no.~4, 907--926.

\bibitem{tanaka2014verified}
Kazuaki Tanaka, Akitoshi Takayasu, Xuefeng Liu, and Shin'ichi Oishi,
  \emph{Verified norm estimation for the inverse of linear elliptic operators
  using eigenvalue evaluation}, Japan Journal of Industrial and Applied
  Mathematics \textbf{31} (2014), no.~3, 665--679.

\bibitem{watanabe2013posteriori}
Yoshitaka Watanabe, Takehiko Kinoshita, and Mitsuhiro Nakao, \emph{A posteriori
  estimates of inverse operators for boundary value problems in linear elliptic
  partial differential equations}, Mathematics of Computation \textbf{82}
  (2013), no.~283, 1543--1557.

\bibitem{watanabe2020some}
Yoshitaka Watanabe, Takehiko Kinoshita, and Mitsuhiro~T Nakao, \emph{Some
  improvements of invertibility verifications for second-order linear elliptic
  operators}, Applied Numerical Mathematics (2020).

\bibitem{yagi2009abstract}
Atsushi Yagi, \emph{Abstract parabolic evolution equations and their
  applications}, Springer Science \& Business Media, 2009.

\end{thebibliography}

\end{document}